\documentclass[reqno]{amsart}

\usepackage{amsfonts,amsthm,amsmath,amssymb,amscd,amsmath,epsf,latexsym,mathrsfs,tkz-euclide,float
}

\usepackage[colorinlistoftodos]{todonotes}

\usepackage{graphicx}
\newtheorem{theorem}{Theorem}
\newtheorem{proposition}[theorem]{Proposition}
\newtheorem{lemma}[theorem]{Lemma}

\newtheorem{remark}{Remark}

\newtheorem{conjecture}{Conjecture}

\newtheorem{problem}{Problem}

\newcommand{\cD}{\mathcal{D}}

\newcommand{\bC}{\mathbb{C}}
\newcommand{\ee}{\end{equation}}

\newcommand\pa {\partial}

\newcommand{\D}{\mathcal D}

\newcommand \PP {\mathcal P}
\newcommand \LL {\mathcal L}

\newcommand \U{\mathcal U} 
\newcommand \V{\mathcal V}

\newcommand \BB {\mathcal B}

\newcommand \fD {\mathfrak D}
\newcommand \I {\mathcal I}
\newcommand \HD {\mathcal HD} 
\newcommand \mO {\mathcal O}

\newcommand {\cM}{\mathcal M}

\begin{document}
          \numberwithin{equation}{section}

          \title
         [Introducing isodynamic points for binary forms and their ratios]{Introducing isodynamic points for binary forms and their ratios}

\author[Ch.~H\"agg]{Christian H\"agg}
\address{Department of Mathematics, Stockholm University, SE-106 91
Stockholm,   Sweden}
\email{hagg@math.su.se}

 \author[B.~Shapiro]{Boris Shapiro}
\address{Department of Mathematics, Stockholm University, SE-106 91
Stockholm,         Sweden and Institute for Information Transmission Problems, Moscow, Russia}
\email {shapiro@math.su.se }

 \author[M.~Shapiro]{Michael Shapiro}
\address{Department of Mathematics, Michigan State University, East Lansing, MI 48824-1027, USA  and National Research University Higher School of Economics, Moscow, Russia}
\email {mshapiro@math.msu.edu}

\begin{abstract}  The isodynamic points  of a plane triangle are known to be the only pair of its centers   invariant under the action of the M\"obius group  $\cM$ on the set of   triangles, \cite{Ki}. Generalizing this  classical result,   
 we introduce below the \emph{isodynamic} map associating  to  a  univariate polynomial  of degree $d\ge 3$ with  at most double roots a polynomial of degree (at most)  $2d-4$ such that this map commutes with the   action of the M\"obius group $\cM$  on the zero loci of the initial polynomial and its image.  The roots of the image polynomial will be called the {\it isodynamic points} of the preimage polynomial. Our construction naturally extends from univariate polynomials  to  binary forms  and  further  to their ratios. 
\end{abstract}

\dedicatory{To Morris Marden, for his contributions to geometry of polynomials}

\maketitle

\section{Introduction}  \label{sec:intro}


One of the classical  problems about triangles in the Euclidean plane is to find all points such that the distances to the vertices are inversely proportional to the lengths of the opposite sides. There are typically two such points called the \textit{first and second isodynamic points} of the triangle under consideration.  (Every equilateral triangle however has just one isodynamic point; the second one can be thought as lying at infinity.)  
  An elementary construction of the isodynamic points using Apollonian circles of a triangle can be found in Figure \ref{classicalIsoPts1} below.\footnote{As we mentioned in the abstract, the isodynamic points can be also described as unique triangle centers  which are invariant under M\"obius transformations.  The formal definition of triangle  centers as well as more information about classical isodynamic points can be found  in the Encyclopedia of Triangles Centers, \cite{Ki}
 and in \S~\ref{sec:classical}.}   
 
 Besides the invariance of the isodynamic points under the action of the M\"obius group on the vertices of triangles, see the footnote, a large number of  their intriguing  properties 
 can be found in two recent publications \cite{Pa,Ra} and references therein.  An interesting old source of information about the isodynamic points is the dissertation \cite{Mo}. 
 
 







 \medskip 
 As we will see below,  the classical geometric recipe  associating to  a plane triangle its isodynamic point(s) can be substituted by  a pleasant explicit formula  \eqref{eq:isoTriPolynomial} associating  to the unique monic cubic polynomial whose roots are the vertices of the triangle an appropriate polynomial of degree at most two vanishing at its isodynamic point(s). This formula seems to be missing in the existing literature.   
 

\medskip
 Inspired by the latter observations,   we consider the following  question. 

\begin{problem}\label{prob:main} {\rm Find a (natural) map from an open and dense subset of complex-valued monic polynomials of a given degree $d\ge 3$  to some space of univariate polynomials  which commutes with the action of the M\"obius group $\cM\simeq PGL_2(\bC)$ on the respective zero loci of  the preimages and of the images.}  
\end{problem}  

\begin{remark}  {\rm Observe that the M\"obius group $\cM$ does not quite act on the space of monic polynomials or on their zero loci  since a M\"obius transformation typically has a pole in $\bC$, i.e., it sends one point in $\bC$ to $\infty$. The same difficulty already occurs in case of isodynamic points of triangles. Therefore, to ensure a well-defined action of the M\"obius group $\cM$ and for our Problem~\ref{prob:main}  to be correctly stated, we will later  instead of the space of monic polynomials of a given degree $d$ consider  the space of homogeneous binary forms of the same degree $d$.     
} 
\end{remark}

\begin{figure}[H]
\begin{center}
\usetikzlibrary {calc,intersections,through}
\begin{tikzpicture}
\clip (-4,3.4) rectangle + (10.5,-5.7); 
\definecolor{triangleColor}{RGB}{0,0,0}
\definecolor{circleColor}{RGB}{0,192,0}
\definecolor{sideExtensionColor}{RGB}{155,50,216}
\definecolor{internalBisectorColor}{RGB}{0,0,255}
\definecolor{externalBisectorColor}{RGB}{255,0,0}

\coordinate (A) at (1,0);
\coordinate (B) at (6,0);
\coordinate (C) at (0,3);
\tkzInCenter(A,B,C)
\tkzGetPoint{ic}
\coordinate (U) at (intersection of A--ic and B--C);
\coordinate (V) at (intersection of B--ic and A--C);
\coordinate [label={[shift={(0.23,-0.465)}]$V$}] (W) at (intersection of C--ic and A--B);
\coordinate (ep1) at ($(U)!15cm!(C)$);
\coordinate (ep2) at ($(V)!12cm!(A)$);
\coordinate (ep3) at ($(W)!7cm!(A)$);
\coordinate (ep4) at ($(U)!5cm!(B)$);

\tkzInCenter(A,C,ep3)
\tkzGetPoint{ic1}
\coordinate (ap1) at ($(A)!15cm!(ic1)$);
\tkzInCenter(A,C,ep1)
\tkzGetPoint{ic2}
\coordinate (ap2) at ($(C)!5cm!(ic2)$);
\tkzInCenter(A,B,ep4)
\tkzGetPoint{ic3}
\coordinate (ap3) at ($(B)!10cm!(ic3)$);
\coordinate [label={[shift={(0.23,-0.465)}]$U$}] (isct1) at (intersection of A--ep3 and C--ap2);
\coordinate (isct2) at (intersection of A--ap1 and C--ep1); 
\coordinate (isct3) at (intersection of A--ep2 and B--ap3); 
\tkzCircumCenter(C,isct1,W)\tkzGetPoint{circ1} 
\draw[internalBisectorColor, line width=0.25mm] (A) -- (ic) -- (intersection of A--ic and B--C);
\draw[internalBisectorColor, line width=0.25mm] (B) -- (ic) -- (intersection of B--ic and A--C);
\draw[internalBisectorColor, line width=0.25mm] (C) -- (ic) -- (intersection of C--ic and A--B);
\draw[color=sideExtensionColor,line width=0.25mm] (C) -- (isct2);
\draw[color=sideExtensionColor,line width=0.25mm] (A) -- (isct3);
\draw[color=sideExtensionColor,line width=0.25mm] (A) -- (isct1);

\draw [externalBisectorColor, line width=0.25mm] (A) -- (isct2); 
\draw [externalBisectorColor, line width=0.25mm] (C) -- (isct1); 
\draw [externalBisectorColor, line width=0.25mm] (B) -- (isct3); 

\tkzDefCircle[circum](C,isct1,W)
\tkzGetPoint{circ1} \tkzGetLength{rayOne}
\tkzDrawCircle[R,color=circleColor,line width=0.25mm](circ1,\rayOne pt)
\tkzDefCircle[circum](A,isct2,U)
\tkzGetPoint{circ2} \tkzGetLength{rayTwo}
\tkzDrawCircle[R,color=circleColor,line width=0.25mm](circ2,\rayTwo pt)
\tkzDefCircle[circum](B,isct3,V)
\tkzGetPoint{circ3} \tkzGetLength{rayThree}
\tkzDrawCircle[R,color=circleColor,line width=0.25mm](circ3,\rayThree pt)
\draw [color=triangleColor, line width=0.5mm] (A) -- (B) -- (C) -- cycle;

\tkzInterCC[R](circ1,\rayOne pt)(circ2,\rayTwo pt) \tkzGetPoints{temp1}{temp2}
\coordinate [label={[shift={(-0.37,-0.12)}]$S'$}] (Iso1) at (temp1);
\coordinate [label={[shift={(-0.35,-0.065)}]$S$}] (Iso2) at (temp2);
\tkzDrawPoint[size=7](Iso1)
\tkzDrawPoint[size=7](Iso2)
\tkzDrawPoint[color=black,fill=white,size=5](U);
\tkzDrawPoint[color=black,fill=white,size=5](V);
\tkzDrawPoint[color=black,fill=white,size=5](W);
\tkzDrawPoint[color=black,fill=white,size=5](isct1);
\tkzDrawPoint[color=black,fill=white,size=5](isct2);
\tkzDrawPoint[color=black,fill=white,size=5](isct3);
\coordinate [label={[shift={(-0.35,-0.40)}]$B$}] (Afix) at (A);
\coordinate [label=right:$C$] (Bfix) at (B);
\coordinate [label=above:$A$] (Cfix) at (C);
\end{tikzpicture}
\caption{Let $U$ and $V$ be the points on sideline $BC$ met by the exterior and interior bisectors of angle $A$. The circle having diameter $UV$ is called the {\it $A$-Apollonian circle}. The $B$- and $C$- Apollonian circles are constructed similarly. Each circle passes through one vertex and both isodynamic points $S$ and $S'$, see \cite{Ki}.}
\label{classicalIsoPts1}
\end{center}
\end{figure}
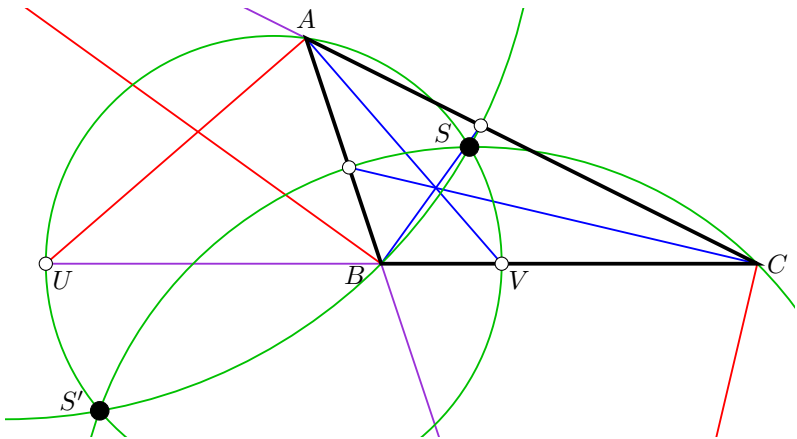  

In what follows, we present one non-trivial solution to Problem~\ref{prob:main} which for $d=3$ sends cubic monic polynomials to polynomials of degree at most two and which,  on the level of their zero loci, associates to a triple of points in the Euclidean plane the above classical isodynamic point(s) of their convex hull, see Proposition~\ref{lm:isodynamicConnection}. We will call this map \emph{isodynamic}. 

\smallskip

More exactly,  denote by $Pol_d$ the affine space of monic complex-valued polynomials of degree $d$  and by $Pol_d^\star\subset Pol_d$ its subset consisting of monic polynomials with roots of multiplicity at most $2$. Further denote by $\PP_n$ the linear space of all complex-valued univariate polynomials of degree at most $n$.   (Observe that $\PP_n$ is isomorphic to the linear space $\BB_n$ of  binary forms of degree $n$.)  Denote by $\BB_n^\star\subset \BB_n$ the subset of homogeneous binary forms with roots of multiplicity at most $2$ (considered as points in $\bC P^1$).    
 
 \smallskip
Given a monic polynomial $P(z)$ of a given degree $d$, consider its \emph{polar derivative}\footnote{Polar derivative has been  already considered by E.~Laguerre and G.~P\'olya jointly with G.~Szeg\"o, who used different terminology.  A nice survey of its properties is Ch.~3 of the famous treatise \cite{Ma}. One the most recent application of polar derivative to geometry of polynomials can be found  in \cite{SeSe}.}
\begin{equation}\label{eq:polar}
\D(z,u)=  d \, P(z)+(u-z) P^\prime(z).
\end{equation}
  Note that $\D(z,u)$ is a bivariate polynomial in $(z,u)$ which is linear non-homogeneous in the variable  $u$ and has degree at most $d-1$ in $z$. The \emph{(univariate) isodynamic map} 
  \begin{equation}\label{eq:isodyn} 
  ID_d: Pol_d^\star \to \PP_{2d-4}
  \end{equation} 
    sends  a monic polynomial $P(z)$  of degree $d$ with  roots of multiplicity at most $2$ to the polynomial $ID_d(P)$ in the variable $u$ given by   
\begin{equation}\label{eq:formula}
ID_d(P) :=\mathrm{Discr} (\D(z,u),z).
\end{equation}
 In other words, $ID_d$ sends $P(z)$ to the discriminant of its polar derivative $\D(z,u)$ with respect to  the variable $z$. The roots of $ID_d(P)$ are, by definition, the {\it isodynamic points} of $P(z)$.
 
\begin{remark}  {\rm 
Recall that for a bivariate polynomial $U(z,u)$,  its discriminant  $\mathrm{Discr} (U(z,u),z)$ with respect to the variable $z$  is defined as follows. Assume that $U(z,u)=a_0(u)z^n+~\dots+a_n(u)$ has degree $n$ in the variable $z$. Then 
$$\mathrm{Discr} (U(z,u),z):= \frac{(-1)^{\binom {n}{2}}} {a_0(u)} \mathrm {Resultant}( U(z,u), U^\prime_z(z,u), z)$$ 
where $\mathrm{Resultant}$ stands for the standard resultant of two polynomials, see e.g. \cite {GKZ}. A similar definition is valid for polynomials in any number of variables.}
\end{remark}

Further let $\I_P$ be the zero locus of $ID_d(P)$ considered as a divisor of degree at most $2d-4$ in the complex plane of the variable $u$. 
 We will call 
$\I_P$ the \emph{(affine) isodynamic  divisor}  and its points the \emph{(affine) isodynamic points}  of $P(z)$.  (In Figure~\ref{classicalIsoPts1} the isodynamic points are shown in the same plane as the zeros of a cubic polynomial $P(z)$.) 

\medskip
We can also formulate a slightly different  recipe of obtaining the isodynamic divisor $\I_P$. Namely, given a polynomial $P(z)=z^d+a_1z^{d-1}+a_2z^{d-2}+\dots + a_d$ of degree $d$ as above, consider the rational function 
\begin{equation}\label{eq:ass1}
R_P(z):=z-\frac{d \, P(z)}{P^\prime(z)}=-\frac{a_1z^{d-1}+2a_2z^{d-2}+\dots + (d-1)a_{d-1}z+da_d}{dz^{d-1}+(d-1)a_1z^{d-2}+\dots + 2a_{d-2}z+a_{d-1}}
\end{equation}
 which we call the \emph{associated rational function}  of $P$ and take the divisor of the critical 
values of $R_P(z)$. One can  show that the latter divisor coincides with $\I_P$, i.e. the critical values of  $R_P(z)$ are the isodynamic points of $P(z)$, see Lemma~\ref{lm:triv} below. Thus the map \eqref{eq:isodyn}   can be thought as a version of the \emph{Lyashko-Looijenga map} sending a monic polynomial of degree $d$  with simple roots  to the univariate polynomial whose zero locus is the set of all critical values of its associated rational function, comp. \cite{Ar}.

\begin{remark}  {\rm There might exist natural extensions of the map $ID_d$ to the space of all polynomials of a fixed degree or another appropriate compactification of $Pol_d^\star$, but we do not pursue this direction of study here.}  
\end{remark} 

\begin{remark}{\rm One intriguing detail of our construction is that the associated rational function  $R_P(z)$ is exactly the one used  in  the  relaxed Newton method for finding roots of $P(z)$ for the relaxation parameter equal to $d$, see e.g. \cite{Gi}. However in our set-up we are interested in the critical values $R_P(z)$ and not in its intersection points with the diagonal line one typically  approaches using  iterations of $R_P(z)$. At the moment we do not see any obvious connection of our problem to complex dynamics of the relaxed Newton method although such connection is quite plausible.  
}
\end{remark}

Let us now modify the above construction to accommodate the case of binary forms and to ensure a well-defined action of the M\"obius group $\cM$ on the preimage and the image spaces.  Given $P(z)$ as above, define its homogenization $P(x,y):= y^d P(\frac {x}{y})$, where $z=\frac{x}{y}$.  
 Using this notation,  in the homogeneous coordinates $(x: y)$ on $\bC P^1$ one has the following alternative expression for the homogenization of the associated rational function 
\begin{equation}\label{eq:ass2}
R_P\left(x,y \right)= - \frac{P^\prime_y(x,y)}{P^\prime_x(x,y)}, 
\end{equation} 
where $R_P\left(x,y \right)$ is obtained by the homogenization of the enumerator and denominator of the above $R_P(z)$, 
see Lemma~\ref{lm:triv}(ii). Noticing that $R_P\left(x,y \right)$ has the same degrees of its numerator and denominator and therefore defines a rational function on $\bC P^1$, 
  we can introduce the \emph{(bivariate) isodynamic map}   
\begin{equation}\label{eq:IDH}
 ID_d: Proj(\BB_d^\star) \to Proj(\BB_{2d-4}) 
\end{equation}
sending a binary form $P(x,y)\in \BB_d^\star$ (considered up to a constant factor) to the binary form $ID_d^P$ (considered up to a constant factor) whose projectivized zero locus is the set of critical values of the rational function
$-\frac{P^\prime_y(x,y)}{P^\prime_x(x,y)}$.  The latter set of critical values is \emph{(projective) isodynamic divisor} and its points are \emph{(projective) isodynamic points}  of $P(x,y)$. As above $\BB_d^\star$ is the space of binary forms of degree $d$ with roots of multiplicity at most $2$ on $\bC P^1$.  

\medskip 
 Our first result is as follows.

\begin{theorem}\label{th:main} The natural action of the M\"obius group  $\cM$ on $Proj(\BB_d^\star)$ and  $Proj(\BB_{2d-4})$ commutes with the bivariate isodynamic map $ID_d$.
\end{theorem}

Note that binary forms of degree $d$ can be thought of as holomorphic sections of the sheaf $\mO(d)$ on $\bC P^1$; this circumstance  sparkles the idea that the above  construction of isodynamic map might extend to meromorphic sections of  $\mO(d)$ as well, i.e. to the ratios of binary forms $\frac{P(x,y)}{Q(x,y)}$, where $\deg P-\deg Q=d$. And indeed, such an extension exists.

\medskip
Namely,   consider the space $B_{d+\pa,\pa}^\star$ of pairs of binary forms $(P(x,y), Q(x,y))$ where $P$ and $Q$ are coprime, both $P$ and $Q$ have roots of multiplicity at most $2$, $\deg P=d+\pa,\; \deg Q=\pa,$ $d>0,$ and $\pa\ge 0$.  We can interpret the pair $(P(x,y), Q(x,y))$ as the bivariate rational function $W(x,y)=\frac{P(x,y)}{Q(x,y)}$. (Note that $W(x,y)$ is  not a rational function on $\bC P^1$ in the usual sense, but is naturally a meromorphic section of $\mO(d)$.)  We can associate to $W(x,y)$ its divisor $D(W)$ of degree $d$ whose positive part, i.e. divisor  of zeros $D^0(W)$ has degree $d+\pa$ and whose negative part, i.e.  divisor  of poles $D_\infty(W)$ has degree $-\pa$. 

Next, let us define the associated rational function of $W$  as 
\begin{equation}\label{eq:asshomrat}
R_{W}(x,y)=-\frac {W^\prime_y(x,y)}{W^\prime_x(x,y)}.
\end{equation}
We will show  that $R_{W}(x,y)$ has  the same degree of its denominator  $W^\prime_x(x,y)$ as of its numerator $W^\prime_y(x,y)$ implying that it is a well-defined rational function on $\bC P^1$ or, equivalently, a meromorphic section of $\mO(0)$.  Namely, for generic binary forms $P(x,y)$ and $Q(x,y)$ as above  the numerator and denominator  both have degrees equal to $2d+4\pa -4$. This circumstance explains why although  formula \eqref{eq:asshomrat} looks the same for all meromorphic sections of $\mO(d)$ it makes sense not only to fix $d,$ but also  the degree $\pa$ of the pole divisor.  

Let us finally associate to the divisor $D(W)$ the positive divisor $D(R_W)$  of all critical values of the associated rational function $R_W$. 


\smallskip
Formula \eqref{eq:asshomrat}   can also be rewritten in a way similar to  \eqref{eq:ass1}. Indeed, restriction of the above section $W(x,y)$ of $\mO(d)$ to the local chart $y=1$ on $\bC P^1$ can be identified with the univariate rational function  
$$w(z)=\frac{p(z)}{q(z)}$$ 
where $p(t)=P(t,1)$ and $q(t)=Q(t,1)$. Typically, $\deg p(t)=\deg P(x,y)=d+\pa$ and $\deg q(t)=\deg Q(x,y)=\pa$. Moreover we can define the associated (univariate)  rational function 
$$R_w(z):=z-\frac{d\, w(z)}{w^\prime(z)}=z-\frac{d\, p(z)q(z)}{p^\prime(z)q(z)-p(z)q^\prime(z)}.$$
One can easily see that $R_w(z)=R_W(x,y)\vert_{y=1}$  and its set of critical values can be defined similarly to \eqref{eq:polar}.

 Namely, for $w(z)=\frac{p(z)}{q(z)}$  with coprime $p(z)$ and $q(z)$ having  roots of multiplicity at most $2$ such that the degree $\deg p=d+\pa$ and $\deg q=\pa$ where $d\ge 1$, define its polar derivative as 
\begin{equation}\label{eq:ratPolarDerivative}
\D(z,u)=d\, w(z)+(u-z)w^\prime(z).
\end{equation}
The numerator of $\D(z,u)$ is given by the expression 
\begin{equation}  \label{eq:resurat}
N\D(z,u)=d \, p(z)q(z)+(u-z)(p^\prime(z) q(z) -p(z) q^\prime(z)).
\end{equation}
We define the \emph{(univariate) rational isodynamic map} 
\begin{equation}\label{eq:map}
ID_{d,\pa}: Rat_{d+\pa,\pa}^\star \to \PP_{2d+4\pa-4}
\end{equation} 
 as sending  a rational function $w(z)$ to the polynomial given by   
\begin{equation}\label{eq:funct}
ID_{d,\pa}(u) := \mathrm{Discr}(N\D(z,u),  z).
\end{equation} 
Here $Rat_{d+\pa,\pa}^\star$  is the space of rational functions $w=\frac{p}{q}$ with coprime pairs of polynomials $(p,q)$ having only roots of multiplicity at most $2$ , $\deg p=d+\pa$ and $\deg q=\pa$ for $\pa\ge 1$. If $\pa=0$ then 
$Rat_{d,0}^\star=Pol^\star_d.$ The fact that $ID_{d,\pa}$ is well-defined will be  proven in Proposition~\ref{lm:defi} below. 

\smallskip
Now define the \emph{rational isodynamic divisor } $\I_w$ of the rational function $w$  as the zero divisor of $ID_{d,\pa}(w)$. 
Similarly to the above, we can show that   $\I_w$ is the divisor of the criticial values of $R_w(z)$. 


\begin{theorem}\label{th:mainRat} In the above notation, the action of the M\"obius group  $\cM\simeq PGL_2(\bC)$ on $Rat_{d+\pa,\pa}^\star$ and $\PP_{2d+4\pa-4}$ commutes with the operation of taking the divisor of isodynamic points, i.e. with the map $ID_{d,\pa}$. 
\end{theorem}

The structure of the paper is as follows. In \S~\ref{sec:proofs} we settle the above mentioned results and discuss several properties of isodynamic maps and their discriminants.  In \S~\ref{sec:ex} we present several explicit formulas for the isodynamic maps for polynomials and rational functions in low degree cases. In \S~\ref{sec:classical} we recall the construction of the classical isodynamic points of triangles and show that they fit as a special case of our general framework.  Finally, in \S~\ref{sec:final} we indicate some  further possible directions of study.

\medskip
\noindent
\emph{Acknowledgements.} The second author wants to thank  Hr.~Sendov for sending him the paper \cite{SeSe}.  He also wants to acknowledge the financial support of his research provided by the grant 2021-04900 of the Swedish Research Council.  Research of the third author was supported by the  NSF grant DMS- 2100791.

\section{Proofs of  basic results and some properties of isodynamic maps} \label{sec:proofs}  

\subsection{Proofs of Theorems \ref{th:main} and \ref{th:mainRat}} 
We start with the following statement.  As above set $w(z)=\frac{p(z)}{q(z)}$ with coprime $p$ and $q$ having roots of multiplicity at most $2$ and $\deg p=d+\pa$, $\deg q=\pa$, $d\ge 1$. Further, let $R_w(z)= z-\frac {d\, w(z)}{w^\prime(z)}$ and $W(x,y)=\frac {P(x,y)}{Q(x,y)}$, where $P(x,y)=y^{d+\pa}p\left(\frac {x}{y}\right)$ and $Q(x,y)=y^{\pa}q\left(\frac {x}{y}\right)$.  Finally, set $\cD(z,u)=d\, w(z)+(u-z)w^\prime(z)$.

\begin{lemma}\label{lm:triv} In the above notation, 

\noindent
{\rm (i)} $u^\ast$ is a value of the variable $u$ such that the polar derivative $\cD(z,u^\ast)$ has a multiple root in $z$ if and only if $u^\ast$ is a critical value of $R_w(z)$;

\noindent
{\rm(ii)}   the associated rational function $R_W(x,y)$ obtained by homogenizing the numerator and denominator of $R_w(z)$ coincides with $-\frac{W_y^\prime(x,y)}{W_x^\prime(x,y)}$. 
\end{lemma}

\begin{proof} To settle (i), note that if $u^\ast$ is a value of the variable $u$ for which the polar derivative $\cD(z,u)$ has a multiple root  in the variable $z$ then denoting this root by $z^\ast$ we get 
$$0=d\, w(z^\ast)+(u^\ast-z^\ast)w^\prime(z^\ast)\Leftrightarrow u^\ast=z^\ast-\frac{d \, w(z^\ast)}{w^\prime(z^\ast)}=R_P(z^\ast).$$
Note that the latter expression is well-defined if $w^\prime (z^\ast)\neq 0$. Since $z^\ast$ is a multiple root of $R_w(z)$, $u^\ast$ has to be a critical value of $R_w(z)$ at $z^\ast$. 

To settle (ii), let us rewrite the ratio $-\frac{W^\prime_y}{W^\prime_x}$.  By definition, $W(x,y)=y^d\frac{p\left(\frac{x}{y}\right)}{q\left(\frac{x}{y}\right)}$. 
Thus 
$$W^\prime_x=y^d\frac{\frac{1}{y}\left(p^\prime\left(\frac{x}{y}\right) q\left(\frac{x}{y}\right) - p\left(\frac{x}{y}\right) q^\prime\left(\frac{x}{y}\right)\right)} {q^2\left(\frac{x}{y}\right) }=y^{d-1} 
\frac{\left(p^\prime\left(\frac{x}{y}\right) q\left(\frac{x}{y}\right) - p\left(\frac{x}{y}\right) q^\prime\left(\frac{x}{y}\right)\right)} {q^2\left(\frac{x}{y}\right) }.$$

Similarly, 
$$W^\prime_y= d\cdot y^{d-1}\frac {p\left(\frac{x}{y}\right) }{q\left(\frac{x}{y}\right) }   +\frac {y^d \left(-\frac{x}{y^2}\right) \left(p^\prime\left(\frac{x}{y}\right) q\left(\frac{x}{y}\right) - p\left(\frac{x}{y}\right) q^\prime\left(\frac{x}{y}\right)\right)}
{q^2\left(\frac{x}{y}\right) }.
$$
The above implies
$$-\frac{W^\prime_y}{W^\prime_x}=-\frac{ d \, p\left(\frac{x}{y}\right) q\left(\frac{x}{y}\right)- \frac{x}{y}\left(p^\prime\left(\frac{x}{y}\right) q\left(\frac{x}{y}\right) - p\left(\frac{x}{y}\right) q^\prime\left(\frac{x}{y}\right)\right)}{p^\prime\left(\frac{x}{y}\right) q\left(\frac{x}{y}\right) - p\left(\frac{x}{y}\right) q^\prime\left(\frac{x}{y}\right)}=\frac{x}{y} - \frac{ d \, p\left(\frac{x}{y}\right) q\left(\frac{x}{y}\right)}{p^\prime\left(\frac{x}{y}\right) q\left(\frac{x}{y}\right) - p\left(\frac{x}{y}\right) q^\prime\left(\frac{x}{y}\right)}
$$
$$=z-\frac {d \, p(z)q(z)}{p^\prime(z)q(z)-p(z)q^\prime(z)},$$
where $z=\frac{x}{y}$.
\end{proof}

Since  Theorem~\ref{th:main} is a special case of Theorem~\ref{th:mainRat} we present below only the proof of the latter result. 

\begin{proof}  Using the above notation, set $R_W(x,y)=-\frac{W^\prime_y}{W^\prime_x}$ and make a change of coordinates $u=\alpha x+\beta y;\; v=\gamma x+\delta y$.  Using the chain rule we obtain   
\begin{equation}
\begin{cases}
W^\prime_x= \alpha W^\prime_u + \gamma W^\prime_v \\
W^\prime_y= \beta W^\prime_u + \delta W^\prime_v
\end{cases}    
\Leftrightarrow 
\hskip0.5cm
\begin{cases}
W^\prime_u= \frac{1}{\fD}\left( \delta W^\prime_x - \gamma W^\prime_y \right)\\
W^\prime_v= \frac{1}{\fD}\left( -\beta W^\prime_x + \alpha \frac{\partial}{\partial y}\right),
\end{cases}  
\end{equation}
where $\fD=\alpha \delta -\beta \gamma$. 
Thus  introducing $R_W(u,v)=-\frac{W^\prime_v}{W^\prime_u}$, we obtain
$$R_W(u,v)= - \frac {-\beta W^\prime_x + \alpha W^\prime_y}{\delta W^\prime_x - \gamma W^\prime_y} =  \frac { \alpha \left(\frac{-W^\prime_y}{W^\prime_x}\right) +\beta}{\gamma \left(\frac{-W^\prime_y}{W^\prime_x}\right) +\delta}= \frac { \alpha R_W(x,y) +\beta}{\gamma R_W(x,y) +\delta}. 
$$
Thus the action of the M\"obius group $\cM$ in the homogeneous coordinates $(x:y)$ results in the same action of $\cM$ on the associated rational functions which implies that the locus of critical values experiences the same action. The result follows.   
\end{proof}  

\subsection{Discriminant of the isodynamic map} 
Below we discuss possible  situations when the image of the isodynamic map $ID_{d,\pa}$ is degenerate. Namely,  we describe  

\smallskip
\noindent
(i) for which pairs $(p(z),q(z))$ the map $ID_{d,\pa}$ is well-defined (respectively, for which polynomials $P(z)$ the map $ID_d$ is well-defined);

\smallskip
\noindent
(ii) for which pairs $(p(z),q(z))$ the corresponding divisor $\I_{p,q}$ has degree less than $2d+4\pa-4$ in $\bC$ (respectively, for which polynomials $P(z)$ of degree $d$ the divisor $\I_{P}$ has degree less than $2d-4$);

\smallskip
\noindent
(iii)  for which pairs $(p(z),q(z))$, the corresponding divisor $\I_{p,q}$  has multiple roots (respectively,    for which polynomials $P(z)$ of degree $d$, the corresponding divisor $\I_P$ has multiple roots). 

\medskip
The latter problem can be reformulated as the question of when the associated rational function $R_P(z)$ does not have $2d-4$ distinct critical values. (In \cite{BoSh} dealing with  a similar situation this set  of parameters is called the \emph{Hurwitz discriminant}.) 

\smallskip In order to answer the latter questions we need an additional statement.

\begin{lemma}\label{lm:multroots} A polynomial pencil $f(z,u)=A(z)+uB(z)$ has a multiple root w.r.t the variable $z$ for each value of the variable $u$ if and only if $A(z)$ and $B(z)$ have a common root which has multiplicity at least $2$ for both $A(z)$ and $B(z)$.
\end{lemma} 

\begin{proof} Denote by $z^\ast(u)$ some multiple root of the pencil $f(z,u)$ for a given value of the variable $u$. It has to satisfy the following system of algebraic equations: 
$$\begin{cases} f(z^\ast(u),u)=0\\
\frac{\partial f}{\partial z} (z^\ast(u),u)=0 
\end{cases} 
 \leftrightarrow \quad 
\begin{cases} A(z^\ast(u))+uB(z^\ast(u))=0\\
A^\prime(z^\ast(u))+uB^\prime(z^\ast(u))=0. 
\end{cases} 
$$
(By the assumptions of the lemma, the set of multiple roots $z^\ast(u)$ is a complex algebraic curve in the space $\bC^2$ with coordinates $(u,z)$ which projects surjectively on the $u$-axis.)
Differentiating the first equation in the latter system w.r.t the variable $u$ we get
$$B(z^\ast(u))+\left(A^\prime(z^\ast(u))+uB^\prime(z^\ast(u))\right)\frac{dz^\ast(u)}{du}=0.$$
Using the second equation from the latter system we conclude that $B(z^\ast(u))=0$ which implies that $z^\ast(u)=z^\ast$ is a constant and this constant is a root of $B(z)$. Thus $z^\ast$ is also a root of $A(z)$. Finally, by our assumptions 
$z^\ast$ is a multiple root of both $A(z)$ and $B(z)$. 
\end{proof} 

\medskip
The following claim now answers the above question (i). 

\begin{proposition}\label{lm:defi}
The map $ID_{d+\pa,\pa}$ applied to a rational function $w(z)=\frac{p(z)}{q(z)}$, $\deg p=d+\pa$, $\deg q=\pa$ given by \eqref{eq:funct} is well-defined if and only if 

either 

\noindent
{\rm(*)}  $\pa>0$,  and the polynomials $p, q$ are coprime and have roots of multiplicity at most $2$;

or 

\noindent
{\rm(**)}  $\pa=0$ {\rm (i.e. $q\equiv 1$)},  and all roots of the polynomial $P$ have  multiplicity at most $2$.
\end{proposition}  

\begin{proof}  To settle (*), let us first show that in case  $\pa>0$, if $p(z)$ and $q(z)$ have a common root $z^\ast$  then $N\D(z,u)$ and $N\D^\prime_z(z,u)$ have a common root w.r.t $z$ for any choice of $u$. Thus  
for such $(p(z),q(z))$, the map $ID_{d+\pa,\pa}$ is not defined.  Indeed, we have 

$$\begin{cases} 
N\D(z^\ast,u)=d\cdot p(z^\ast)q(z^\ast)+(u-z^\ast)(p^\prime(z^\ast)q(z^\ast)-p(z^\ast)q^\prime(z^\ast))\equiv 0\\
N\D^\prime_z(z^\ast,u)=(d-1) p^\prime(z^\ast)q(z^\ast)+(d+1) p(z^\ast)q(z^\ast)+(u-z^\ast)(p^{\prime\prime}(z^\ast)q(z^\ast)\\
-p(z^\ast)q^{\prime\prime}(z^\ast))\equiv 0, 
\end{cases} 
$$
i.e., both $N\D(z^\ast,u)$ and $N\D^\prime_z(z^\ast,u)$ vanish identically in the variable $u$.  Exactly the same argument holds if either $p$ or $q$ has a root $z^\ast$ of multiplicity exceeding $2$.

Let us now show the converse implication, i.e., that $ID_{d+\pa,\pa}$ is well-defined if $p(z)$ and $q(z)$ are coprime and have no roots of multiplicity bigger than $2$.  
Assume that $ID_{d+\pa,\pa}$ is not defined  for the pair $(p,q)$ which  is equivalent to the fact that $N\D(z,u)$ has a multiple root w.r.t $z$  for all values of  $u$. Observe that $N\D(z,u)$ is a polynomial pencil of the form $F(z)+u\Phi(z)$ where  $F(z)=d\, p q - z(p^{\prime}q-pq^{\prime})$ and $\Phi(z)=(p^{\prime} q-p q^{\prime})$.  Thus by Lemma~\ref{lm:multroots}, the univariate polynomials $F(z)$ and $\Phi(z)$ must have a common root of multiplicity at least $2$ at some point $z^\ast$.    But if $F(z)$ and $\Phi(z)$ have a multiple common root at $z^\ast$ then $pq$ and $p^{\prime} q-p q^{\prime}$ have a multiple common root at $z^\ast$ as well. But the latter claim is impossible since $p$ and $q$ are coprime and therefore $pq$ has no multiple roots. 

To settle (**) let us first show that if $P(z)$ has a root $z^\ast$ of multiplicity at least $3$, then $\D(z,u)=  d \, P(z)+(u-z) P^\prime(z)$ has a multiple root w.r.t. $z$ for any choice of $u$. Indeed,
$$
\begin{cases} 
\D(z^\ast,u)=d\, P(z^\ast)+(u-z^\ast)P^\prime(z^\ast)\equiv 0,\\
\D^\prime_z(z^\ast,u)=(d-1)\, P^\prime(z^\ast)+(u-z^\ast)P^{\prime\prime}(z^\ast)\equiv 0.
\end{cases}
$$
since $P$ has a least a triple root at $z^\ast$.  Now assume the converse, i.e., that $\D(z,u)$ has a multiple root w.r.t $z$ for any value of $u$. Since  $\D(z,u)$  is  a polynomial pencil of the form $d P(z)-z P^\prime(z) + u P^\prime(z)$,  Lemma~\ref{lm:multroots} implies that $P(z)$ and $P^\prime(z)$ must at least have a common double root $z^\ast$, i.e., $P(z)$ has a least a triple root. 
\end{proof}  

\smallskip
Next define  the polynomial families $p(z)=z^{d+\pa}+a_1z^{d+\pa-1}+\dots + a_{d+\pa}$ and $q(z)=z^\pa+b_1z^{\pa-1}+\dots + b_\pa$. Using these families, consider the expression \eqref{eq:resurat} and the map \eqref{eq:funct}. The following Lemma settles question (ii). 

\begin{lemma}\label{lm:leading}  
{\rm (i)} For $\pa>0$, in the above notation and up to a constant factor, the expression for the map  $ID_{d,\pa}$  has a factor $\mathrm{Resultant}(p,q)$, where  $\mathrm{Resultant}(p,q)$ stands for the resultant of  $p$ and $q$. The leading coefficient of $ID_{d,\pa}$, i.e. the coefficient of $u^{2d+4\pa-4}$ equals $\mathrm{Discr}(p^\prime q - p q^\prime)$, where $\mathrm{Discr}(p^\prime q - p q^\prime)$ stands for the discriminant of $\Phi(z)=p^\prime q - p q^\prime$. (The latter discriminant factors into $\mathrm{Resultant}(p,q)$ and an additional  factor, see Remark~\ref{rem:imp}.) 

\smallskip
\noindent
{\rm (ii)} For $\pa=0$ the leading coefficient of  $ID_{d}$ equals the discriminant $\mathrm{Discr}(P^\prime)$ of $P^\prime(z)$. In other words, the polynomial $ID_{d} (P)$ has degree less than $2d-4$ if and only if the derivative of the initial polynomial $P$ has multiple zeros. 
\end{lemma}  


\begin{proof}   To settle (i),  observe that  by Proposition~\ref{lm:defi},  if  $p(z)$ and $q(z)$ have a common root $z^\ast$ then the bivariate polynomial $N\D(z,u)$ given by \eqref{eq:map} has a multiple root at $z^\ast$ in the variable $z$ for all values of the second variable $u$ which implies that the map $ID_{d,\pa}$ vanishes identically. To show that the leading coefficient of  $ID_{d,\pa}$ equals $\text{Discr}(p^\prime q - p q^\prime)$ let us calculate  $ID_{d,\pa}$ using the standard determinantal formula for the resultant of the Sylvester matrix, see e.g. \cite {GKZ}, ch. 12.  In our situation each non-vanishing entry  of the Sylvester matrix will be a linear non-homogeneous polynomial in $u$ with the leading coefficient and the constant term being polynomials in $z$. If we drop all the constant terms and just keep the terms containing $u$ in the Sylvester matrix, then the determinant of this matrix will be equal to $u^{2d+4\partial -4}$ times the discriminant of the coefficient of $u$ in the original expression $N\D(z,U)$. Since this coefficient equals $p^\prime q- q^\prime p$, the claim follows. 

The argument for (ii) is exactly the same as for (i).  
\end{proof}   

\begin{remark} {\rm Observe that Lemma~\ref{lm:leading}  explains why in the classical situation there exists only one isodynamic point for a triple of non-collinear points in the plane if and only if they form an equilateral triangle; see \S~\ref{sec:ex} below.}
\end{remark}

To solve question (iii), let us first discuss a more general set-up. Assume that we consider a family $\Phi_\lambda(z)=\frac{U_\lambda(z)}{V_\lambda(z)}$ of univariate rational functions depending on some multi-dimensional parameter $\lambda$ taking values in some connected complex  algebraic variety $\Lambda$. We  assume that 

\noindent
(a) univariate polynomials $U_\lambda(z)$ and $V_\lambda(z)$ are coprime for generic values of $\lambda \in \Lambda$, but they can have a common factor for  special values of $\lambda$ belonging to some complex algebraic subvariety of $\Lambda$;

\noindent
 (b)  for generic $\lambda$, $\Phi_\lambda(z)$ has distinct and simple critical values. 
 
 \smallskip By our assumptions, the number of distinct critical values of $\Phi_\lambda(z)$ is constant for almost all $\lambda\in \Lambda$. Denoting this number by $\kappa$, let us  
define the Lyashko-Looijenga map $\LL: \Lambda \to \bC^\kappa$  sending every point $\lambda \in \Lambda$ to the divisor of the critical values of $\Phi_\lambda(z)$. 
We define the \emph{Hurwitz discriminant}    $\HD\subset \Lambda$    
as the set of all  $\lambda \in \Lambda$ for which the divisor $\LL(\lambda)$ does not consist of $\kappa$  simple points. 
According to \cite{BoSh},  set-theoretically, the Hurwitz discriminant $\HD$ typically  contains three irreducible components $\HD^0\cup \HD^W \cup \HD^M$, where  
$${\rm(i)}\quad \HD^0 =\{\lambda\in \Lambda \vert \exists z^\ast \; \text{such that}\;U_\lambda(z^\ast)=V_\lambda(z^\ast)=0\};$$

\noindent
{\rm(ii)}\quad  $\HD^W$ is the closure of $\HD^W_o$ where 
$$\HD^W_o=\{\lambda\in \Lambda\setminus\HD^0 \vert \exists z^\ast \; \text{such that}\; W_\Lambda(U;V) \; \text{has (at least) a double root at}\;z^\ast\}, $$
and $W_\Lambda(U;V)=U_\lambda^\prime(z)V_\lambda(z)- U_\lambda(z)V_\lambda^\prime(z)$ is the Wronskian of $U_\lambda(z)$ and $V_\lambda(z)$; 

\smallskip
\noindent
{\rm(iii)}\quad  $\HD^M$ is the closure of $\HD^M_o$ where
$$\HD^M_o=\{\lambda\in \Lambda \vert \exists z_1\neq z_2 \; \text{such that}\; \Phi_\lambda^\prime(z_1)=\Phi_\lambda^\prime(z_2)=0 \; \text{and}\;  \Phi_\lambda(z_1)=\Phi_\lambda(z_2)  \}.$$

\begin{remark}\label{rem:imp}  {\rm The union $\HD^0\cup \HD^W$ considered as a subset of $\Lambda$  coincides  with the zero locus of the discriminant of the  Wronskian  $W_\Lambda(U;V)$ with respect to $z$.  The third irreducible component $\HD^M$ is usually referred to as the \emph{ Maxwell stratum}, see e.g. \cite{LZ}.}  
\end{remark}  

In our specific situation the family of rational functions $R_P(z)$ under consideration is given by \eqref{eq:ass1} where $P(z)$ runs over the space $Pol_d$. Set 
{\small
\begin{equation}\label{eq:same}
\begin{cases}
\U_P(z)=a_1z^{d-1}+2a_2z^{d-2}+\dots+(d-1)a_{d-1}z+da_d=zP^\prime(z)-d\, P(z); \\
\V_P(z)=dz^{d-1}+(d-1)a_1z^{d-2}+\dots+2a_{d-2}z+a_{d-1}=P^\prime(z)
\end{cases}
\end{equation} }
 i.e., $\U_P(z)$ is the numerator and $\V_P(z)$ is the denominator of the associated rational function $R_P(z)$. 
 
 Set $\D_d := \mathrm{Resultant}\left(ID_d, \frac{\partial ID_d^\prime}{\partial u}, u\right)$.  Numerical experiments with Mathematica strongly support the following guess. 
 
 \begin{conjecture}\label{conj:mult} {\rm (a)} In the above notation, for any $d\ge 3$, $\D_d=(\D_d^0)^{j_d^0} (\D_d^W)^{j^W_d} (\D_d^M)^{j^M_d},$
 where $\D_d^0, \D_d^W, \D_d^M$ are irreducible polynomials whose zero loci satisfy the above conditions {\rm (i), (ii), (iii)} respectively and $j_d^0, j^W_d, j^M_d$ are some non-negative multiplicities. 
 
 \smallskip
 \noindent
 {\rm (b)} The non-negative multiplicities  $j_d^0, j_d^W, j_d^M$ are all positive for $d\ge 5$, see \S~\ref{sec:ex}. 
 
 \end{conjecture}  

\begin{lemma}\label{lm:triv2} In the above notation, $\U_P(z)$ and $\V_P(z)$ are linearly dependent, i.e., $R_P(z)$ is a constant if and only if $P(z)=(z+t)^d$  for some constant $t$.
\end{lemma}

\begin{proof} Direct calculation of proportionality. \end{proof}

\begin{lemma}\label{lm:zero} In the above notation,  $\U_P(z)$ and $\V_P(z)$ have a common zero if and only  $P(z)$ has a multiple root, i.e., the discriminant $\D_d^0$  coincides with the discriminant of $P(z)$.  
\end{lemma}

\begin{proof} 
Indeed, if $\U_P(z^*)=\V_P(z^*)=0$ then $P(z^*)=P^\prime(z^*)=0$ which means that $P(z)$ has a multiple root at $z^*$. 
\end{proof} 

\begin{remark} {\rm Equation~\eqref{eq:same}  implies that 
for a given degree $d$, the Wronskian $W(\U_P,V_P)$ is given by $(d-1)(P^\prime)^2-d P P^{\prime\prime}$. The latter expression  has previously occurred in  
\cite{Lo} in connection with the stronger form of Laguerre inequality, see also \cite{FK}.} 
\end{remark}




Note that the definition \eqref{eq:ratPolarDerivative} of polar derivative of rational functions hints the following definition of a more general polar derivative
\begin{equation}\label{eq:alphaPolarDerivative}
D^\alpha_f(z,u) = \alpha f(z) + (u-z)f'(z)
\end{equation}
where $f(z)$ belongs to some suitable class of functions and $\alpha$ is a real parameter. This expression appears naturally e.g. in the calculation of the asymptotic root-counting measure of the sequence of polynomials $\{(P^n)^{(\lfloor\alpha n\rfloor)}\}$, see \cite{BoHaSh}. In our situation, for $\alpha=1$ and polynomials $P(z)$ of degree $d$, we can define the degree $d-2$ polynomial $\mathcal{W}_P(u) = \mathrm{Discr}(D^1_P(z,u),z) / P(u)$ with zeros $u_1,\dots,u_{d-2}$, which has the following connection to the isodynamic points of $P$.
\begin{lemma}
\label{lm:IsodynamicCentroidCentroid}
For any triple of non-collinear points $z_1,z_2,z_3$, i.e. for $d=3$,  $$u_1 = \left(\frac{2}{3}\right)^2\frac{(z_1+z_2+z_3)^3 - 3^3 z_1 z_2 z_3}{\mathrm{Discr}(P'(z),z)}$$ is the centroid of the isodynamic points $\mathcal{I}_1,\,\mathcal{I}_2$ and the mass center $(z_1+z_2+z_3)/3$. Furthermore, $u_1$ is listed as $X(26613)$ in the Encyclopedia of Triangle Centers \cite{Ki}, and is the Dao-6-point-circle-inverse of $(z_1+z_2+z_3)/3$.
\end{lemma}
\begin{proof}
Trivial calculations.
\end{proof}
Additionally, if $m = m(\alpha)$ denotes the centroid of the zeros of $\mathrm{Discr}(D^\alpha_P(z,u),z)$, then for cubic polynomials $P$, $m(\alpha)$ coincides with  several known triangle centers in \cite{Ki} for various values of $\alpha\in\mathbb{Z}$, see Table \ref{table:knownCentroids}. (The connection between $m(\alpha)$ and these points can  be proved by converting their corresponding trilinear or absolute barycentric coordinates in \cite{Ki} to complex numbers using a reference triangle with vertices $z_1,\,z_2,$ and $z_3$. Mathematica files containing these calculations are available upon request.)

\begin{table}
\begin{tabular}{| p{0.7cm} | p{1.8cm} |}
\hline
$\alpha$ & $m$ \\ \hline
-12 & $X(316)$ \\
-6 & $X(31173)$ \\
-3 & $X(625)$ \\
-1 & $X(10150)$ \\
0 & $X(2)$ \\
2 & $X(5215)$ \\
3 & $X(187)$ \\
4 & $X(26613)$ \\
6 & $X(187)$ \\
24 & $X(14712)$ \\
\hline
\end{tabular}
\vskip0.4cm
\caption{Known centroids of the zeros of $\mathrm{Discr}(D^\alpha_P(z,u),z)$ for cubic $P$, for $\alpha = -50,-49,\dots,50$.}
\label{table:knownCentroids}
\end{table}
\begin{remark} {\rm 
Due to the considerable number of known triangle centers in Table \ref{table:knownCentroids}, and the fact that $m(\alpha)$ traverses a real line in $\mathbb{C}$ as $\alpha$ runs over the real line for any polynomial $P$ of degree $d\ge 2$ with simple zeros, the authors think that this line is worth further study.}
\end{remark}

\section{Examples of isodynamic maps} 
\label{sec:ex} 

To illustrate our results and in order to give some intuition about the objects of our study,  we present below explicit formulas for the isodynamic maps and their discriminants for polynomials and rational functions of low degrees. 

\subsection{Isodynamic maps for polynomials}

\subsubsection{Cubics}  
In the classical case $d=3$, direct calculations give the following explicit formula for the isodynamic map  
\begin{equation}\label{eq:isoTriPolynomial}
ID_3: z^3+az^2+bz+c \; \mapsto \; (a^2-3b)u^2+(ab-9c)u+b^2-3ac.
\end{equation}
Equation \eqref{eq:isoTriPolynomial} implies that 
if the image $ID_3(P)$ of a cubic polynomial $P$ is linear, i.e., $a^2=3b$ then $P(z)=(z+a/3)^3+c-a^3/27$ which means that its zero locus is necessarily an equilateral triangle. 

\smallskip
In the above notation, the discriminant $\D_3$ of the family $ID_3(P)$ w.r.t. the variable $u$ equals $$\D_3=27c^2+(4a^3-18ab)c+4b^3-a^2b^2$$ which is exactly the discriminant $\D_3^0$ of the original polynomial family $P(z)=z^3+az^2+bz+c$ w.r.t. the variable $z$. In other words, two isodynamic points of a triple of points in the plane coincide if and only if at least two of the points in the triple are equal. Thus for $d=3$, $j_3^0=1$ and the discriminants $\D_3^W$ and $\D_3^M$ are empty, see Conjecture~\ref{conj:mult}.

\subsubsection{Quartics}  To simplify our formulas, note that the affine shift $z\to z+ t$ acts trivially on  the isodynamic map $ID_4$.  Using this fact, let us restrict our considerations to  the standard reduced polynomial family 
$$P_4(z)=z^4+az^2+bz+c$$ 
which is the classical versal deformation of the singularity $z^4$.  The isodynamic map 
restricted to the latter family is explicitly given by
{\small
\begin{equation}\label{eq:ID4}
ID_4: z^4+az^2+bz+c \mapsto (32a^3+108b^2)u^4+(72a^2b+864bc)u^3+
(-4a^4+108ab^2-288a^2c+1728c^2)u^2$$ $$+(-4a^3b+108b^3-432abc)u -9a^2b^2+32a^3c.
\end{equation}  }

Up to a constant factor, the leading coefficient $(32a^3+108b^2)$ in the right-hand side of \eqref{eq:ID4} is the discriminant of the derivative of $P_4(z)$ with respect of $z$.

\smallskip
The discriminant of the   right-hand side  of \eqref{eq:ID4}  w.r.t. $u$ is given by 
{\small
\begin{equation}\label{eq:discrr}  
\D_4=(2a^3+27b^2-72ac)^6(-4a^3b^2-27b^4+16a^4c+144ab^2c-128a^2c^2+256c^3).
\end{equation}  
}
The second factor    $-4a^3b^2-27b^4+16a^4c+144ab^2c-128a^2c^2+256c^3$ of the right-hand side of \eqref{eq:discrr} is $\D_4^0$ which is the standard discriminant  of the above family $P_4(z)$. The first factor $2a^3+27b^2-72ac$ is the Wronski discriminant $\D^W_4$. (Both discriminants are quasi-homogeneous with weights $w(a)=2$, $w(b)=3$, $w(c)=4$ of the variables. The total weight of $\D_4$ equals $12$; the total weights of  $\D_4^0$ and of $\D_4^W$ equal $6$.) The Maxwell discriminant $\D^M_4$ is empty and $j_4^0=1, \; j_4^W=6$. Figure~\ref{fig0} shows these discriminants separately and together.  
\begin{figure}

\begin{center}
\includegraphics[scale=0.3]{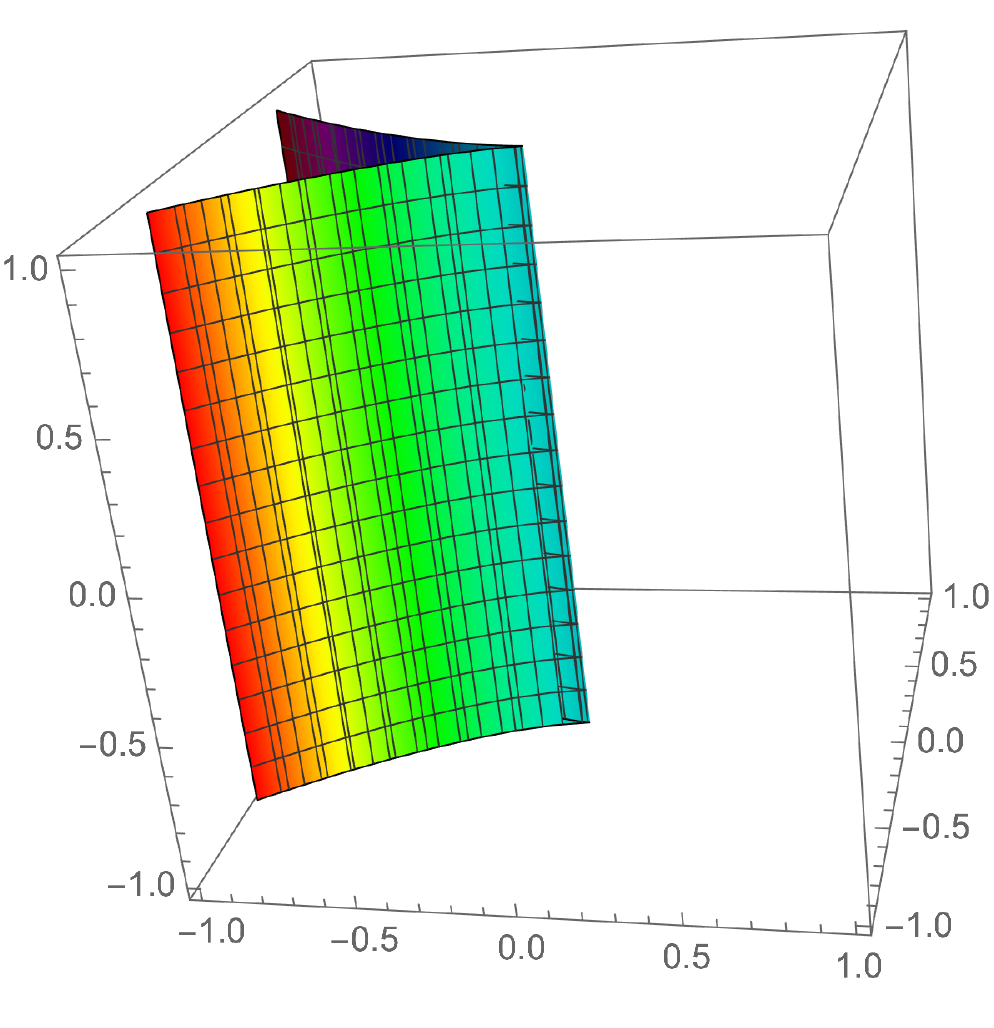}
\includegraphics[scale=0.3]{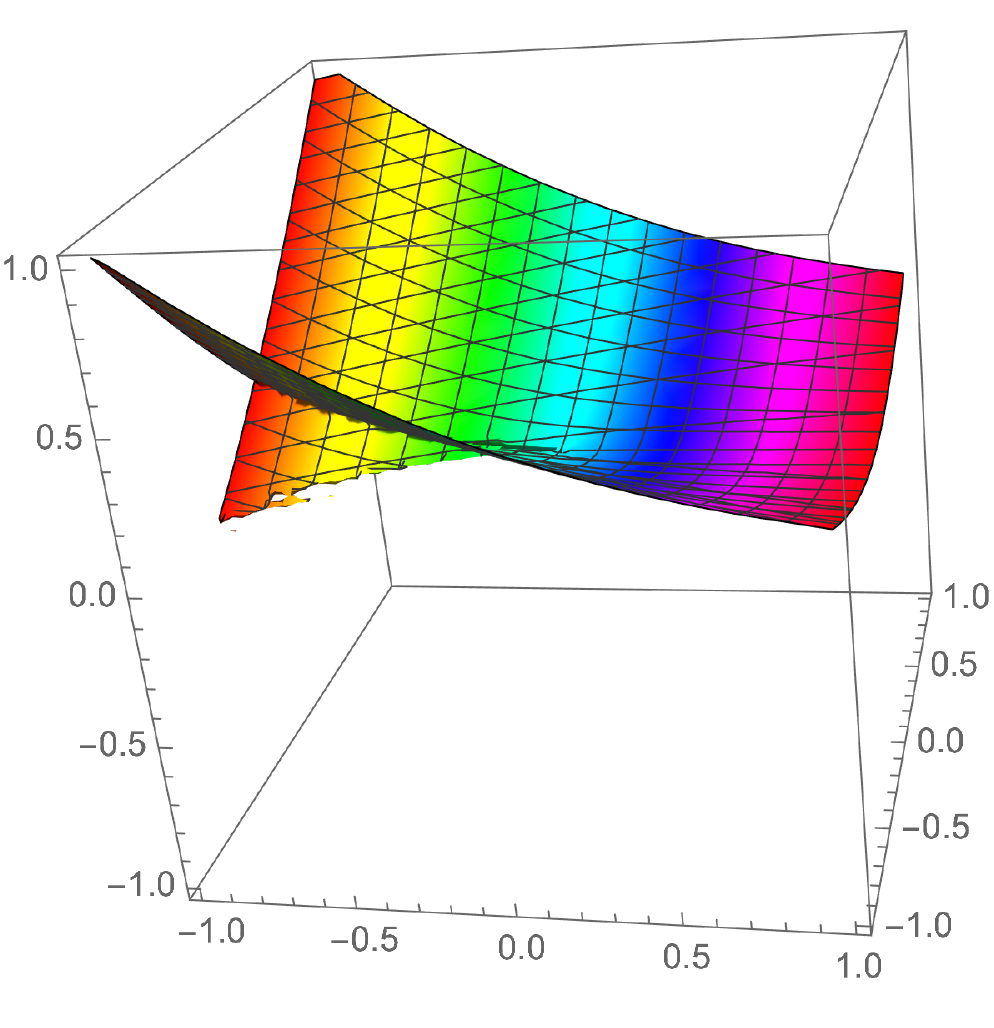}
\includegraphics[scale=0.3]{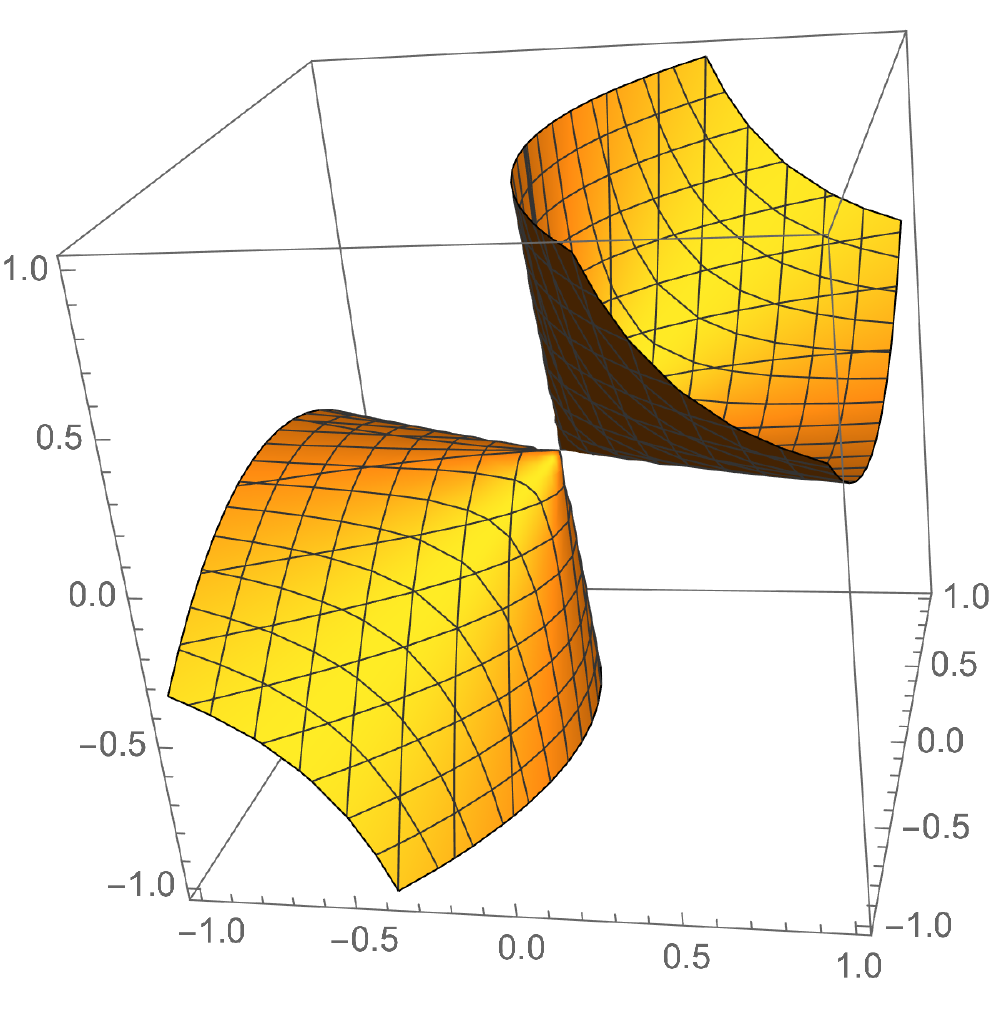}
\includegraphics[scale=0.5]{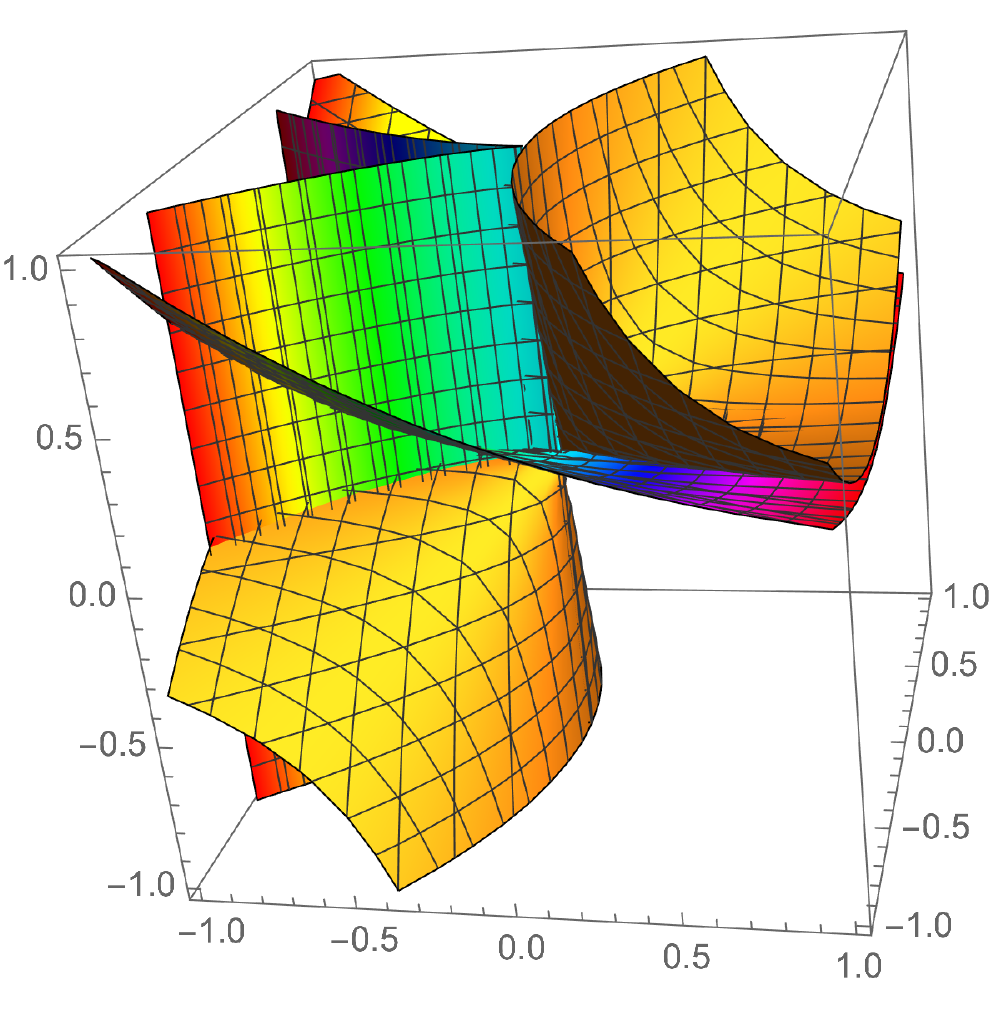}
\end{center}

\caption{Discriminants of the reduced isodynamic family $ID_4$ and their union (bottom). The top-left subfigure is the discriminant of $P^\prime$, the top-central subfigure is $\D_4^0$, i.e. the discriminant of $P$, and the top-right subfigure is $\D_4^W$. (The vertical axis is $c$, the horizontal one is $a$ and the one pointing northeast is $b$.)}
\label{fig0}
\end{figure}




\subsubsection{Quintics}  
Here we present the formulas for $ID_5$ and its discriminant  $\D_5$ since $d=5$ is the minimal value for which $\D_d$ contains all three irreducible components including the Maxwell stratum $\D_d^M$. (All formulas were calculated in Mathematica.)

\smallskip
For $P(z)=z^5+az^3+bz^2+cz+e$, one has 

{\tiny
$$ID_5= (-405 a^4 c+135 a^3 b^2+1800 a^2 c^2-2700 a b^2 c+675 b^4-2000 c^3)u^6-16(-2025 a^4 e-180 a^3 b c+135 a^2 b^3$$ 
$$+18000 a^2 c e -13500 a b^2 e -6000 a b c^2 +2700 b^3 c-30000 c^2 e)u^5 + (27 a^5 c-9 a^4 b^2-3600 a^3 b e+1800 a^3 c^2-600 a^2 b^2 c$$ 
$$+45000 a^2 e^2+135 a b^4-6000 a b c e-8400 a c^3-13500 b^3 e+7200 b^2 c^2-150000 c e^2)u^4 + (135 a^5 e-9 a^4 b c-2 a^3 b^3+3600 a^3 c e $$  
$$-3000 a^2 b^2 e+1560 a^2 b c^2-720 a b^3 c +120000 a b e^2-30000 a c^2 e+135 b^5-36000 b^2 c e+10800 b c^3-250000 e^3)u^3 $$ 
$$+ (135 a^4 b e-117 a^4 c^2+51 a^3 b^2 c-13500 a^3 e^2 -9 a^2 b^4+13200 a^2 b c e-2040 a^2 c^3-3600 a b^3 e +180 a b^2 c^2+60000 a c e^2$$ 
$$+135 b^4 c+45000 b^2 e^2-54000 b c^2 e+10800 c^4)u^2 +  (-270 a^4 c e+135 a^3 b^2 e+12 a^3 b c^2-9 a^2 b^3 c-13500 a^2 b e^2 + 600 a^2 c^2 e$$ $$+9000 a b^2 c e-2160 a b c^3-2025 b^4 e+540 b^3 c^2)u  + (675 a^4 e^2-540 a^3 b c e+128 a^3 c^3+135 a^2 b^3 e-36 a^2 b^2 c^2). 
$$
}

Standard discriminant: 
{\tiny
$$\D_5^0=
108 a^5 e^2-72 a^4 b c e+16 a^4 c^3+16 a^3 b^3 e-4 a^3 b^2 c^2-900 a^3 c e^2+825 a^2 b^2 e^2+560 a^2 b c^2 e-128 a^2 c^4-630 a b^3 c e
$$ 
$$+144 a b^2 c^3-3750 a b e^3+2000 a c^2 e^2+108 b^5 e-27 b^4 c^2+2250 b^2 c e^2-1600 b c^3 e+256 c^5+3125 e^4. 
$$
}

Wronskian discriminant: 
{\tiny
$$ \D_5^W= 675 a^{10} e^2-450 a^9 b c e-8 a^9 c^3+100 a^8 b^3 e+83 a^8 b^2 c^2-36000 a^8 c e^2-36 a^7 b^4 c+19500 a^7 b^2 e^2+22400 a^7 b c^2 e+640 a^7 c^4+4 a^6 b^6$$ 
$$-17100 a^6 b^3 c e-4400 a^6 b^2 c^3-150000 a^6 b e^3+620000 a^6 c^2 e^2+2650 a^5 b^5 e+3920 a^5 b^4 c^2-540000 a^5 b^2 c e^2-328000 a^5 b c^3 e-19200 a^5 c^5$$ 
$$-10000000 a^5 e^4-1110 a^4 b^6 c+157500 a^4 b^4 e^2+394000 a^4 b^3 c^2 e+77600 a^4 b^2 c^4+18000000 a^4 b c e^3-3200000 a^4 c^3 e^2+100 a^3 b^8-162000 a^3 b^5 c e$$ 
$$-88000 a^3 b^4 c^3-5500000 a^3 b^3 e^3-5000000 a^3 b^2 c^2 e^2+640000 a^3 b c^4 e+256000 a^3 c^6+19500 a^2 b^7 e+42800 a^2 b^6 c^2+2700000 a^2 b^4 c e^2$$ 
$$+920000 a^2 b^3 c^3 e-480000 a^2 b^2 c^5-25000000 a^2 b^2 e^4+20000000 a^2 b c^2 e^3-4000000 a^2 c^4 e^2-9000 a b^8 c-150000 a b^6 e^2-700000 a b^5 c^2 e$$ 
$$+224000 a b^4 c^4+30000000 a b^3 c e^3-32000000 a b^2 c^3 e^2+11200000 a b c^5 e-1280000 a c^7+675 b^{10}+90000 b^7 c e-28000 b^6 c^3-10000000 b^5 e^3$$
$$+11000000 b^4 c^2 e^2 -4000000 b^3 c^4 e+480000 b^2 c^6.
$$
}

Maxwell discriminant: 
{\tiny
$$\D_5^M=
9 c^2 a^6-3375 e^2 a^5-6 b^2 c a^5+b^4 a^4-380 c^3 a^4+2325 b c e a^4+40 b^2 c^2 a^3+30000 c e^2
   a^3-525 b^3 e a^3+4400 c^4 a^2-26250 b^2 e^2 a^2+15 b^4 c a^2$$ 
   $$-17000 b c^2 e a^2-4800 b^2 c^3 a+125000 b e^3 a-50000 c^2 e^2a+19500 b^3 c e a-8000 c^5+900 b^4 c^2-75000 b^2 c e^2-3375 b^5 e+50000 b c^3 e. 
$$}

Finally, $\D_5=\D_5^0 (\D_5^W)^3 (\D_5^M)^2.$ In other words, $j_5^0=1,\; j_5^W=3,\; j_5^M=2$.

\subsection{Isodynamic maps for rational functions}   

\smallskip
\noindent
\subsubsection{Case 1-1} 
 Let us consider the  isodynamic map for the simplest nontrivial family of rational functions $w(z)=\frac{p(z)}{q(z)}$ where $p(z)=z^2+az+b,\; q(z)=z+c$. (In this case $d=\pa=1$.) Then 
$$ID_{1,1}: \frac{z^2+az+b}{z+c}\mapsto  4( b-ac+c^2)(u^2+au+b)
$$
and its discriminant equals
$$\D_{1,1}=16(a^2-4b)(b-ac+c^2)^2.$$

\subsubsection{Case 2-1} 
Consider $w(z)=\frac{p(z)}{q(z)}$ where $p(z)=z^3+az^2+bz+c,\; q(z)=z+e$. (In this case $d=2$ and $\pa=1$.)  Then
{\small
$$ID_{2,1}: \frac{z^3+az^2+bz+c}{z+e}\mapsto  4 (e^3-c + b e - a e^2)(  (-a^3 + 27 c + 9 a^2 e - 27 b e) u^4 + (-6 a^2 b + 54 a c + 8 a^3 e - 
    18 a b e - 54 c e) u^3$$ 
    $$  + (18 a^2 c -12 a b^2 + 54 b c + 12 a^2 b e - 
    18 b^2 e - 54 a c e) u^2  + (18 a b c-8 b^3  + 54 c^2 + 6 a b^2 e - 
    54 b c e) u$$ $$ -9 b^2 c + 27 a c^2 + b^3 e - 
 27 c^2 e ). 
$$}

The  discriminant of  $ID_{2,1}$  is given by 
$$\D_{2,1}=-1289945088 (-a^2 b^2 + 4 b^3 + 4 a^3 c - 18 a b c + 27 c^2)^4 (c - 
   b e + a e^2 - e^3)^8. $$
  
  \subsubsection{Case 1-2}  
   Consider $w(z)=\frac{p(z)}{q(z)}$ where $p(z)=z^3+az^2+bz+c,\; q(z)=z^2+ez+f$. (In this case $d=1$ and $\pa=2$.)  Then 
{\tiny
$$ID_{1,2}: \frac{z^3+az^2+bz+c}{z^2+ez+f}\mapsto -16 (c + b u + 
   a u^2 + u^3)  (c^2 - b c e + a c e^2 - c e^3 + b^2 f - 2 a c f - a b e f + 
   3 c e f + b e^2 f + a^2 f^2 - 2 b f^2 - a e f^2 + f^3)$$ 
   $$  ((-b^3 + 27 c^2 + 3 a b^2 e - 27 b c e - 
    3 a^2 b e^2 + 27 a c e^2 + a^3 e^3 - 27 c e^3 + 18 b^2 f - 
    54 a c f - 9 a b e f + 81 c e f - 9 a^2 e^2 f + 27 b e^2 f + 
    27 a^2 f^2 - 81 b f^2) u^3 $$ 
    $$+ (-9 b^2 c + 
    27 a c^2 + 3 b^3 e - 9 a b c e - 6 a b^2 e^2 + 18 a^2 c e^2 + 
    3 a^2 b e^3 - 27 a c e^3 + 24 a b^2 f - 54 a^2 c f - 54 b c f - 
    21 a^2 b e f + 18 b^2 e f + 135 a c e f - 3 a^3 e^2 f$$ 
    $$ + 
    9 a b e^2 f + 27 a^3 f^2 - 72 a b f^2 - 81 c f^2 - 9 a^2 e f^2 + 
    27 b e f^2) u^2 + (-9 b^2 c e + 27 a c^2 e - 3 b^3 e^2 + 9 a b c e^2 + 
    3 a b^2 e^3 - 27 b c e^3 + 27 b^3 f - 72 a b c f - 81 c^2 f$$ 
    $$ - 
    21 a b^2 e f + 18 a^2 c e f + 135 b c e f - 6 a^2 b e^2 f + 
    18 b^2 e^2 f + 24 a^2 b f^2 - 54 b^2 f^2 - 54 a c f^2 + 
    3 a^3 e f^2 - 9 a b e f^2 - 9 a^2 f^3 + 27 b f^3) u -9 b^2 c e^2 + 27 a c^2 e^2$$ 
    $$ + b^3 e^3 - 27 c^2 e^3 + 27 b^2 c f - 
 81 a c^2 f - 9 a b c e f + 81 c^2 e f - 3 a b^2 e^2 f + 
 27 b c e^2 f + 18 a^2 c f^2 - 54 b c f^2 + 3 a^2 b e f^2 - 
 27 a c e f^2 - a^3 f^3 + 
 27 c f^3  ). 
$$}

The discriminant of $ID_{1,2}$ equals
{\tiny  
$$\D_{1,2}=21641687369515008 (-a^2 b^2 + 4 b^3 + 4 a^3 c - 18 a b c + 
   27 c^2) (e^2 - 4 f)\times $$ 
   $$(b^4 - 6 a b^2 c + 9 a^2 c^2 - a b^3 e + 
   3 a^2 b c e + 9 b^2 c e - 27 a c^2 e + b^3 e^2 + a^3 c e^2 - 
   9 a b c e^2 + 27 c^2 e^2 + 3 a^2 b^2 f - 10 b^3 f - 10 a^3 c f + 
   $$ 
   $$36 a b c f - 27 c^2 f - a^3 b e f + 3 a b^2 e f + 9 a^2 c e f - 
   27 b c e f + a^4 f^2 - 6 a^2 b f^2 + 9 b^2 f^2)^2 (c^2 - b c e + 
   a c e^2 - c e^3 + b^2 f - 2 a c f - a b e f$$ 
   $$ + 3 c e f + b e^2 f + 
   a^2 f^2 - 2 b f^2 - a e f^2 + f^3)^12 (-2 b^3 c + 9 a b c^2 - 
   27 c^3 + b^4 e - 3 a b^2 c e - 9 a^2 c^2 e + 27 b c^2 e - 
   a b^3 e^2 + 6 a^2 b c e^2$$ 
   $$ - 9 b^2 c e^2 + b^3 e^3 - a^3 c e^3 - 
   a b^3 f + 6 a^2 b c f - 9 b^2 c f + 6 b^3 e f - 6 a^3 c e f + 
   a^3 b e^2 f - 6 a b^2 e^2 f + 9 a^2 c e^2 f + a^3 b f^2 - 
   6 a b^2 f^2 $$ 
   $$+ 9 a^2 c f^2 - a^4 e f^2 + 3 a^2 b e f^2 + 
   9 b^2 e f^2 - 27 a c e f^2 + 2 a^3 f^3 - 9 a b f^3 + 27 c f^3)^6. $$}  

\section{Appendix I. Classical isodynamic points for plane triangles in our context}\label{sec:classical}

As we mentioned in the introduction classical isodynamic points for plane triangles have been studied for more than a century and for their properties are described in e.g. \cite{Wik, Ra, Pa}. The main result of this appendix is the following  statement  which relates the construction of the present paper to one of the definitions of the classical isodynamic points. 

\begin{proposition}\label{lm:isodynamicConnection}
The zeros of $ID_3(P)$ given by \eqref{eq:isoTriPolynomial} are the first and second isodynamic points of the triangle $T\subset\mathbb{C}$ whose vertices are the (noncollinear) zeros of $P(z) = z^3 + az^2 + bz + c$.
\end{proposition}
\begin{proof}
Let $z_1,\,z_2$ and $z_3$ denote the zeros of $P(z)$. We will show that the zeros $u_1,\,u_2$ of $ID_3(P)$ satisfy the equations
\begin{equation}\label{eq:isodynamicProdEq}
\vert u-z_1\vert \vert z_3 - z_2\vert = \vert u-z_2\vert \vert z_3 - z_1\vert = \vert u-z_3\vert \vert z_2 - z_1\vert
\end{equation}
which is one of the definitions for the isodynamic points of $T$. Due to symmetry of the vertices of $T$, it is sufficient to show that the first of these equations is satisfied. Furthermore, we can restrict ourselves to the case $z_1 = 0,\,z_2 = 1$ and $z_3 = \rho\in\{\zeta\in\mathbb{C}\mid\mathrm{Im}(\zeta)>0\}$ since these points are the vertices of any triangle $T\subset\mathbb{C}$ under the action of affine transformations.

Under these conditions we have that
$$u_1 = \frac{\rho(\rho+1)-\sqrt{-3(\rho(\rho-1))^2}}{2+2\rho(\rho-1)}$$
and
$$u_2 = \frac{\rho(\rho+1)+\sqrt{-3(\rho(\rho-1))^2}}{2+2\rho(\rho-1)}.$$
Thus, for $u_1$ we need to show that
\begin{equation}\label{eq:isodynamicRho}
\vert u_1 - z_1\vert\vert z_3-z_2\vert = \vert u_1 - z_2\vert\vert z_3-z_1\vert,
\end{equation}
or, equivalently (for all $\rho\neq\frac{1\pm\sqrt{3}i}{2}$; the pole $\rho=\frac{1+\sqrt{3}i}{2}$ can be seen to satisfy equation \eqref{eq:isodynamicRho}),
\begin{equation}\label{eq:isodynamicRho2}
\left\vert\frac{(u_1 - z_1)(z_3-z_2)}{(u_1 - z_2)(z_3-z_1)}\right\vert = 1 \iff \left\vert \frac{\left(\frac{\rho(\rho+1)-\sqrt{-3(\rho(\rho-1))^2}}{2+2\rho(\rho-1)} - 0\right)(\rho-1)}{\left(\frac{\rho(\rho+1)-\sqrt{-3(\rho(\rho-1))^2}}{2+2\rho(\rho-1)} - 1\right)(\rho-0)} \right\vert = 1.
\end{equation}
Equation \eqref{eq:isodynamicRho2} simplifies to
\begin{equation}\label{eq:isodynamicRho3}
\left\vert\frac{\sqrt{3}\rho(\rho-1)}{\sqrt{-(\rho(\rho-1))^2}}-1\right\vert = 2
\end{equation}
which is equivalent to $\vert \pm\sqrt{-3}-1\vert=2$. Since $\pm\sqrt{-3}=\pm i\sqrt{3}$ the later fact is trivial. Calculations for $u_2$ are completely similar.

\end{proof}





\section{ Final remarks} \label{sec:final}

\smallskip
\noindent
{\bf 1.}  The following analog of the famous theorem of E.~Laguerre, see e.g. \S~13 of Ch. 3 in \cite{Ma}, about the location of the roots of the polar derivative in our current setting is supported by our numerical experiments with polynomials of degrees $\ge 3$ with randomly distributed roots in various rectangles.

\begin{conjecture}\label{th:Laguerre} For a univariate polynomial $P$  of degree $d\ge 3$ with at most double roots, no circle or line in $\bC$  separates the zero locus of $P(z)$ from $\I_P$. 
\end{conjecture}

\smallskip
\noindent
{\bf 2.} Our numerical experiments with the isodynamic points for the Legendre and the Laguerre polynomials resulted in the following intriguing pictures which need to be explained, see Figure \ref{fig:isodynamic} below.

\begin{figure}[htp]
\begin{center}
\includegraphics[width=.36\textwidth]{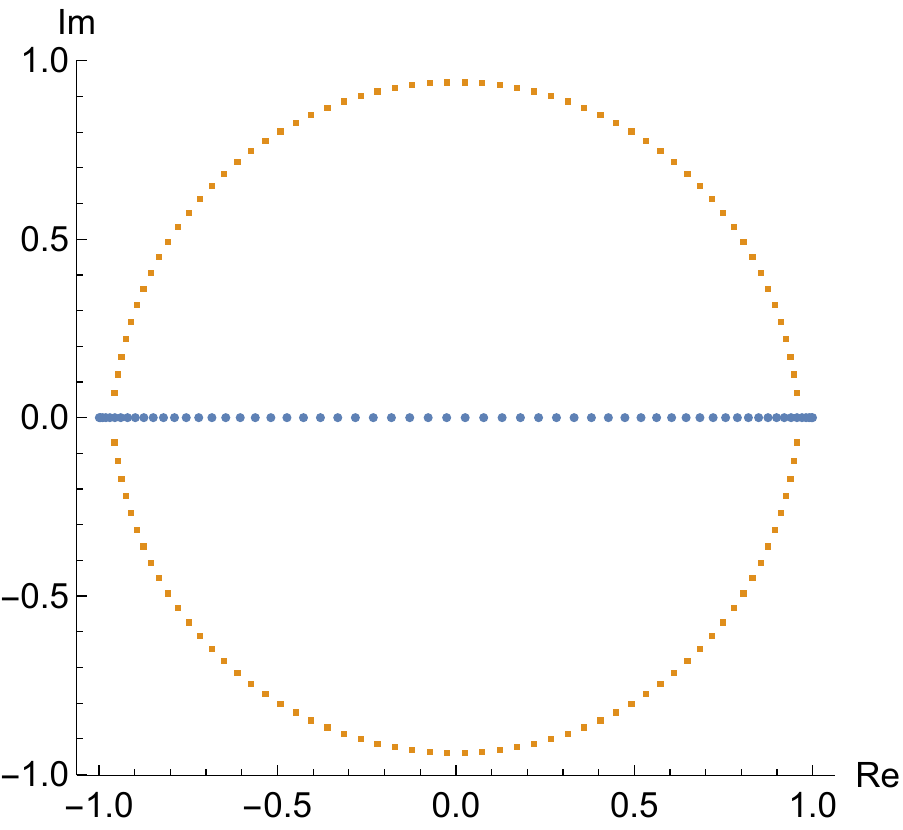}\hfill
\includegraphics[width=.62\textwidth]{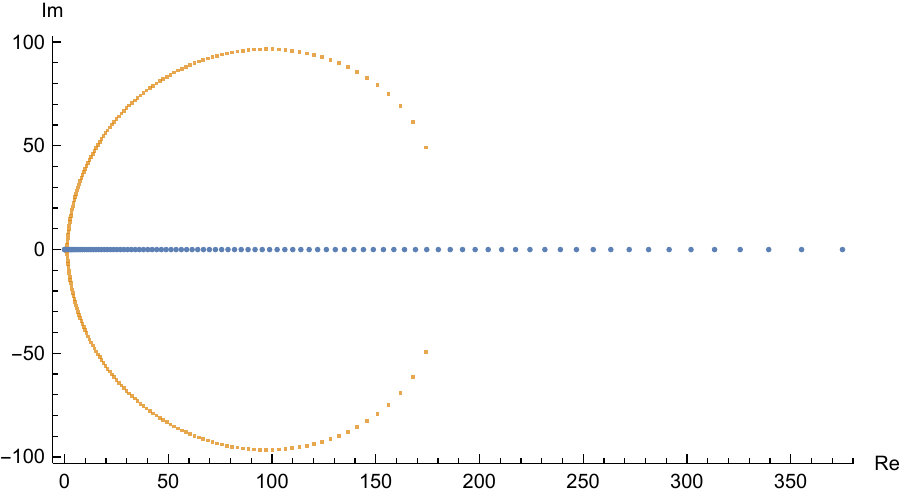}
\end{center}
\caption{The zeros of the 60th Legendre polynomial (left, blue) and the 100th Laguerre polynomial (right, blue), shown along with the isodynamic points of these polynomials (brown).}
\label{fig:isodynamic}
\end{figure}

\smallskip
\noindent
{\bf 3.}  Is our construction of the isodynamic map unique in some appropriate sense?

\smallskip
\noindent
{\bf 4.}  Is it possible to find a natural algebraic -- geometric interpretation of the divisor of isodynamic points?  Taking into account that the linear spaces $\BB_n$ of degree $n$ binary forms are the standard irreducible representation of  $sl_2(\bC)$ is there a representation -- theoretic meaning of our constructions. 

\smallskip
\noindent
{\bf 5.} Prove Conjecture~\ref{conj:mult} and its analog for rational functions. 



\end{document}